\documentclass[reqno, 12pt]{amsart}
\usepackage{amsmath,amssymb, amsthm}
\usepackage{hyperref}
 \newtheorem{theorem}{Theorem}[section]
\newtheorem{proposition}{Proposition}[section]

\newtheorem{lemma}{Lemma}[section]
 \newtheorem{remark}{Remark}[section]
\usepackage{hyperref}

\usepackage{color}

\begin{document}

\def\eps{\varepsilon}

\title[Liouville-type theorem]
{A Liouville-type theorem for the $3$-dimensional parabolic Gross-Pitaevskii and related systems} 

 \author[Q. H. Phan]{Quoc Hung Phan}
 
 \address{Institute of Research and Development, Duy Tan University, Da Nang, Vietnam}
 \email{phanquochung@dtu.edu.vn }

 \author[Ph. Souplet]{Philippe Souplet }
 
 \address{Universit\'e Paris 13, Sorbonne Paris Cit\'e,
CNRS UMR 7539, Laboratoire Analyse, G\'{e}om\'{e}trie et Applications,
93430 Villetaneuse, France}
\email{souplet@math.univ-paris13.fr}

\begin{abstract}
We prove a Liouville-type theorem for semilinear parabolic systems of the form  
$${\partial_t u_i}-\Delta u_i =\sum_{j=1}^{m}\beta_{ij} u_i^ru_j^{r+1},  \quad i=1,2,...,m$$
 in the whole space ${\mathbb R}^N\times {\mathbb R}$. Very recently, Quittner [{\em Math. Ann.}, DOI 10.1007/s00208-015-1219-7 (2015)] has established an optimal result for $m=2$ in dimension $N\leq 2$, and partial results in higher dimensions in the range $p< N/(N-2)$. By nontrivial modifications of the techniques of Gidas and Spruck and of Bidaut-V\'eron, we partially improve the results of Quittner in dimensions $N\geq 3$. In particular, our results solve the important case of the parabolic Gross-Pitaevskii system -- i.e.~the cubic case $r=1$ -- in space dimension $N=3$, for any symmetric $(m,m)$-matrix $(\beta_{ij})$ with nonnegative entries, positive on the diagonal.
By moving plane and monotonicity arguments,  that we actually develop for more general cooperative systems,
we then deduce a Liouville-type theorem in the half-space  ${\mathbb R}^N_+\times {\mathbb R}$.
As applications, we give results on universal singularity estimates, universal bounds for global solutions, and blow-up rate estimates for the corresponding initial value problem.
\end{abstract}

\maketitle
\section{Introduction}
In this article, we study the semilinear parabolic system of the form 
\begin{align}\label{c51}
 \frac{\partial u_i}{\partial t}-\Delta u_i =\sum_{j=1}^{m}\beta_{ij} u_i^ru_j^{r+1}, \qquad x\in \Omega,\quad t\in I, \quad i=1,2,...,m,
\end{align}
 where $r>0$,  $\Omega$ is a domain of ${\mathbb R}^N$,   $I$ is an interval of $\mathbb R$, $m\ge 2$ is an integer 
 and $B=(\beta_{ij})$ is a real $m\times m$ symmetric matrix. We assume throughout, unless otherwise specified, that 
 \begin{align}\label{cooperativecond}
   \beta_{ij}\geq 0 \; \text{ for all } \; i\ne j, \quad \text{and } \beta_{ii}>0 \text{ for all } i.
 \end{align}

The system~(\ref{c51})  can be  used to describe heat propagation in a $m$-component combustible mixture \cite{Beb89}, in this case $u_i$ represent the temperatures of the interacting components. 

In the special case $r=1$, system~(\ref{c51}) can be seen as a parabolic counterpart of the $m$-coupled nonlinear Schr\"odinger system
\begin{align}\label{Schrodinger}
 \frac{1}{\sqrt{-1}}\frac{\partial u_i}{\partial t}-\Delta u_i =\sum_{j=1}^{m}\beta_{ij} u_i|u_j|^{2},  \quad i=1,2,...,m.
\end{align}
The cubic system (\ref{Schrodinger}), also known as the Gross-Pitaevskii system, arises in mathematical models for various phenomena in physics, such as  nonlinear optics and  the Hartree-Fock theory for Bose-Einstein condensation (see e.g.  \cite{BDW10, DWW10, WW08, DWZ11}). In nonlinear optics,  the solution $u_i$ stands for the $i$-th component of the beam in Kerr-like photorefractive media (see e.g. \cite{AA99}). The positive constant $\beta_{ii}$ is for self-focusing in the $i$-th component of the beam. The coupling constant $\beta_{ij}\  (i\ne j)$ is the interaction between the $i$-th and the $j$-th components of the beam. In the theory of Bose-Einstein condensation (see \cite{Hio99}), $u_i$ are the corresponding condensate amplitudes, $\beta_{ii}$ and $\beta_{ij}$ are the intraspecies and interspecies scattering lengths. The case $\beta_{ij}\geq 0$ means that  the interactions of states $|i\rangle$ and $|j\rangle$ are attractive.
\smallskip

System~(\ref{c51}) has been recently studied in various mathematical directions such as: the local and global existence \cite{Ama85, Mor89}, 
H\"older regularity \cite{DWZ11}, symmetry property \cite{Pol09, FP09}, blow-up behavior \cite{MZ00},  and Liouville-type theorems \cite{Pha14, Qui14, MZ00}. Our main goal in this paper is to prove Liouville-type theorems  for the  problem (\ref{c51}) and then to deduce their important applications on qualitative properties of solutions.

\smallskip
We recall that  Liouville-type theorems are statements about the nonexistence of solution  in the entire space or in half-space. In recent years, the Liouville property has been refined considerably and has emerged as one of the most powerful tools in the study of initial and boundary value problems for nonlinear PDEs. It turns out that one can obtain from Liouville-type theorems a variety of results on qualitative properties of solutions such as: universal, pointwise, a priori estimates of local solutions; universal and singularity estimates; decay estimates; universal bound of global solutions, initial and final blow-up rates, etc..., see \cite{PQS07, PQS07b} and references therein. In addition, it was shown in \cite{WW08} that the parabolic system (\ref{c51}) can be used in the study of solutions of the corresponding elliptic problems, provided one can show suitable a priori bound of the global solutions of (\ref{c51}). This a priori bound property is a consequence of the Liouville-type theorems. 
\smallskip

Let us recall the elliptic counterpart 
\begin{align}\label{c52}
 -\Delta u_i =\sum_{j=1}^{m}\beta_{ij} u_i^ru_j^{r+1}, \quad x\in {\mathbb R}^N, \quad i=1,2,...,m.
\end{align}
This system has attracted much attention of mathematicians in recent years, especially for the cubic case $r=1$, see e.g.  \cite{QS11, GLW14, TTVW, DW12, DWW10, MZ08} for more references. Concerning the Liouville property, it is well known that the Liouville-type  result for (\ref{c52})  plays an important role in the study elliptic problems as well. 
The optimal Liouville-type theorem for nonnegative solutions of (\ref{c52}) was completely proved by Reichel and Zou \cite{RZ00} (see also \cite{GL08}) via moving sphere techniques, under  the optimal Sobolev subcritical range $p<p_S$, where
$$p:=2r+1$$ 
and 
\begin{align*}
p_S:=\begin{cases}
\frac{N+2}{N-2}\quad &\text{if}\quad N\geq 3,\\
\infty \quad &\text{if}\quad N=1,2.
\end{cases}
\end{align*}
We note that, if we remove the positivity assumption on the diagonal of $B$, then problem (\ref{c52}), and hence  (\ref{c51}), may have many semi-trivial solutions due to system collapsing, i.e., solutions with one or more components being zero.  
For example, if $\beta_{11}=0$ then $U=(C, 0,...,0)$ is a solution of problem (\ref{c52}). And indeed, the study of the system (\ref{c52}) becomes more delicate due to the existence of semi-trivial solutions, especially in the applications of Liouville-type theorems (cf.~\cite{BVR96, MSS14, LLW15} and see Remark~\ref{remsemitrivial} below). 

\smallskip
For the corresponding parabolic problem (\ref{c51}), the Liouville property is much less understood. 
Recall that, even in the scalar case, i.e. for the classical nonlinear heat equation
\begin{equation}\label{NLH}
u_t-\Delta u=u^p,
\end{equation}
it is by now not completely settled. Indeed, the Liouville property for (\ref{NLH}) is conjectured to be true under the optimal
condition $1<p<p_S$ (and this is known to hold in the class of radiallly symmetric solutions \cite{PQ06, PQS07}),
but this has been proved so far only under the stronger restriction  $1<p<p_B$ \cite{BV98}, where
\begin{equation}\label{defBV}
p_B(N):=\begin{cases}
\frac{N(N+2)}{(N-1)^2}\quad &\text{if}\quad N\geq 2,\\
\infty \quad &\text{if}\quad N=1,
\end{cases}
\end{equation}
or, very recently \cite{Qui14}, for $N=2$.
One of the main difficulties is that the techniques of moving planes or moving spheres do not work as in the elliptic case \cite{RZ00}. 
As for system~(\ref{c51}) 
under assumption (\ref{cooperativecond}), only some partial cases are known:  
\begin{itemize}
\item First, one can easily obtain a scalar parabolic inequality $\partial_tz-\Delta z\geq Cz^p$ for $z=\sum_{i=1}^m u_i$ and deduce the Fujita-type result of problem (\ref{c51}), namely there is no nontrivial nonnegative solution in ${\mathbb R}^N\times \mathbb R_+$ if   $1<p\leq 1+\frac{2}{N}$.  
\vskip 2pt

\item In the case $m=2$,  the Liouville-type theorem for (\ref{c51}) has been recently  proved in \cite{Pha14} in dimension $N=1$ and for radial solutions in any dimension if $p<p_S$. 
More recently, Quittner~\cite{Qui14} has proved the optimal Liouville-type theorem  in dimensions $N\leq 2$, and has also given a partial result in  dimension $N\geq 3$ under the condition $p< \frac{N}{N-2}$. The main tools in \cite{Qui14} are scaling argument  and energy estimates. 
\vskip 2pt

\item Under the assumption $m=2$, $\beta_{ii}=0$ for $i=1,2$ and $\beta_{12}>0$, the Liouville-type theorem for positive solutions of problem (\ref{c51}) can be shown via comparison technique (see \cite{Qui15} and \cite{QS12, MSS14} for the elliptic case). 
More precisely, by taking the difference of the two equations and suitably using the maximum principle, we may show that  $u=v$ and thus reduce the system to a scalar equation.  
\end{itemize}

In this paper, we shall use a different approach to establish a Liouville-type theorem for problem~(\ref{c51}) in the whole space in a larger range of $p$ and for any $m$. 
We shall then treat Liouville-type theorems in the half-space by reduction to the whole space case. 

\section{Liouville type results}

 Our main result in the whole space case is the following. 
 
\begin{theorem}\label{th1} Let $m\ge 2$, $r>0$ and assume $p:=2r+1<p_B(N)$, where $p_B(N)$ is defined in (\ref{defBV}). 
Let $B$ satisfy (\ref{cooperativecond}). Then  system (\ref{c51}) has no nontrivial, nonnegative classical solution in ${\mathbb R}^N\times {\mathbb R}$.
\end{theorem}

\begin{remark} \label{rm1} 
a) We stress that $p_B(N)>\frac{N}{N-2}$ when $N\geq 3$, and  our result is a partial improvement of Quittner~\cite{Qui14} in higher dimensions. 
In particular, it solves  the important case of the parabolic Gross-Pitaevskii (cubic) system where $r=1$ and $N=3$. 

 \smallskip
b) If $N=2$ then, by \cite{Qui14}, the conclusion of Theorem~\ref{th1} is actually true for all $r>0$ (the result in \cite{Qui14} is formulated only in the case $m=2$ but the proof is valid for any $m\geq 2$).

\smallskip
c) Our proof of Theorem~\ref{th1} relies on nontrivial modifications of the technique developed by Bidaut-V\'eron \cite{BV98} for the scalar nonlinear heat equation. The latter was an adaptation of the celebrated method of Gidas and Spruck~\cite{GS81a} for elliptic equations (see also \cite{BVR96} for some particular elliptic systems). It is based on nonlinear integral estimates and Bochner formula. This technique is completely different from that of \cite{Qui14}, which relies on scaling and energy arguments. 

\smallskip
d) In \cite{MZ00}, Merle and Zaag proved Liouville-type theorems for the so-called ancient solutions of the system
\begin{align}\label{MZeq}
U_t-\Delta U=F(|U|)U,
\end{align}
with $F(|U|) \sim |U|^{p-1}$ as $|U|\to \infty$, under the assumption $p<p_S$ and 
\begin{align}\label{MZ}
u_i(x, t) \leq  C(T -t)^{-1/(p-1)}.
\end{align}
Namely they showed that any solution of (\ref{MZeq}) in ${\mathbb R}^N\times (-\infty,T)$ which satisfies  
(\ref{MZ}) is independent of the space variable. 
In the special case $F(|U|)=|U|^2$ ($p=3$),  Proposition~\ref{c5th5} below (which is a consequence of Theorem~\ref{th1}) guarantees that (\ref{MZ}) holds if $3<p_B(N)$. Thus,  the estimate (\ref{MZ}) is always true if $N\le 3$.
\end{remark}

 We now turn to the case of a half-space ${\mathbb R}^N_+=\{x\in {\mathbb R}^N: x_1>0\}$.

\begin{theorem}\label{Liouhalf}
Let $r\geq 1$ and assume either $N\le 3$ or $N=4$ and $p=2r+1<p_B(3)=15/4$. Let $B$ satisfy (\ref{cooperativecond}).
Then the problem
\begin{align}\label{halfspace}
\begin{cases}
\displaystyle \frac{\partial u_i}{\partial t}-\Delta u_i &=\displaystyle\sum_{j=1}^{m}\beta_{ij} u_i^ru_j^{r+1}, \qquad x\in {\mathbb R}^N_+,\quad t\in \mathbb R, \quad i=1,2,...,m,\\
\noalign{\vskip 1mm}
u_i&=0, \qquad x\in \partial {\mathbb R}^N_+,\quad t\in \mathbb R, \quad i=1,...,m,
\end{cases}
\end{align}
has no nontrivial, nonnegative, bounded classical solution.
\end{theorem}

\begin{remark}
The upper restrictions on $N$ in  Theorem~\ref{Liouhalf} are consequences of the condition $r\geq 1$,
which is required by the use of the moving plane method in the proof of Theorem~\ref{Liouhalf}.
\end{remark}

 Theorem~\ref{Liouhalf} is a consequence of the following result which can be applied for more general cooperative parabolic systems. 

\begin{theorem}\label{general}
Consider the following system
\begin{align}\label{coop}
\begin{cases}
\displaystyle \frac{\partial u_i}{\partial t}-\Delta u_i&=f_i(u_1,...,u_m), \quad x\in {\mathbb R}^N_{+}, \; t\in {\mathbb R}, \; i=1,...,m,\\
u_i&=0, \quad  x\in \partial{\mathbb R}^N_{+}, \; t\in {\mathbb R}, \;  i=1,...,m.
\end{cases}
\end{align}
Assume that $f_i:[0,\infty)^m\to {\mathbb R}$ are  $C^1$-functions satisfying: 
\begin{align*}
&(H_1) \quad f_i(0,...,0)=0, \quad i=1,...,m,\\
&(H_2)\quad \frac{\partial f_i}{\partial u_j}(u_1,...,u_m)\geq 0, \quad \text{for all } (u_1,...,u_m)\in [0,\infty)^m \ \text{ and all $i\ne j$}, \\
&(H_3) \quad  \sum_{j=1}^m 
\frac{\partial f_i}{\partial u_j}(0,...,0) {\, \le\ } 0, \quad \text{for all } i,\\
\noalign{\noindent and that}
&(H_4) \quad \text{ any nontrivial, nonnegative, bounded solution of (\ref{coop}) is positive in ${\mathbb R}^N_{+}\times{\mathbb R}$.}
\end{align*}
Then any nontrivial, nonnegative,  bounded solution $U=(u_i)$  of (\ref{coop}) is increasing in~$x_1$:
\begin{align}\label{monotonex1}
\frac{\partial u_i}{\partial x_1}(x,t)>0, \; x\in {\mathbb R}^{N}_{+},\; t\in {\mathbb R}, \; i=1,...,m.
\end{align}
\end{theorem}

Theorem~\ref{general} is an analogue for systems of  \cite[Theorem 2.4 (c1)]{PQS07b} for scalar equations. The proof follows the idea in \cite{PQS07b} which is based on a moving plane technique. 
However, significant additional difficulties arise in the case of systems. This leads to the introduction of the assumption $(H_4)$ (which, in turn, is necessary for 
the conclusion (\ref{monotonex1})  to hold). Then, in order to deduce Theorem~\ref{Liouhalf} from Theorem~\ref{general}, we use an induction argument on the number $m$ of components.
\smallskip

The outline of the rest of the paper is as follows.  
In section 3, we give applications of our Liouville-type theorems, namely universal singularity estimates, including initial and final blowup estimates, as well as universal bounds for global solutions. Section~4 is then devoted to the proof of the Liouville-type Theorem~\ref{th1} in the whole space, and section~5 to the proofs of Theorems~\ref{Liouhalf} and \ref{general} in a half-space.  Finally, a version of the maximum principle for cooperative systems, suitable to our needs, is given in Appendix.

\section{Applications of Liouville-type results}

As a first application of Theorem \ref{th1}, we obtain universal singularity estimates
in time and space, including universal initial and final blowup estimates in the case~$\Omega={\mathbb R}^N$.

\begin{proposition}\label{c5th5}
  Let $r>0$ and assume $p:=2r+1<p_B(N)$ if $N\geq 3$. Let $B$ satisfy~(\ref{cooperativecond}).
  There exists a universal constant $C=C(N,p, B)>0$ such that,  for any domain $\Omega$ of ${\mathbb R}^N$
  and  any nonnegative classical solution $U$ of (\ref{c51}) in $\Omega\times(0,T)$,
there holds
\begin{align}\label{universal}
u_i(x,t)\leq C\left (t^{-1/(p-1)}+ (T-t)^{-1/(p-1)}+ \text{\rm dist}^{-2/(p-1)}(x,\partial \Omega) \right), \;
\end{align}
for all $(x,t)\in \Omega\times(0,T)$ and $i=1,...,m$.
In particular, in the case $\Omega={\mathbb R}^N$, we have
\begin{align}\label{universal2}
u_i(x,t)\leq C\left (t^{-1/(p-1)}+ (T-t)^{-1/(p-1)} \right).
\end{align}

\end{proposition}

Note that Proposition~\ref{c5th5} covers in particular the Gross-Pitaevskii case $r=1$ in dimension $N=3$.
The proof is based on a reduction to Theorem~\ref{th1} by rescaling and doubling  arguments. 
Since it is completely similar to that in \cite[Theorem 3.1]{PQS07b} for the scalar case, it is therefore omitted.
We stress that the proof only requires Theorem~\ref{th1} for {\it bounded} solutions
(and Theorem~\ref{th1} in the general case is finally obtained as a consequence of Proposition~\ref{c5th5}). 

\begin{remark} \label{remsemitrivial}
(a) It is clear that if the positivity condition of the diagonal of $B$ is removed, then the universal singularity estimate (\ref{universal})
fails due to the existence of arbitrary large semi-trivial solutions. Indeed, if $\beta_{11}=0$, then for any $A>0$,
the constant function $U=(A,0,...,0)$ is a semi-trivial solution of problem (\ref{c51}) in any domain $\Omega\times (0,T)$.

\smallskip
(b) In the elliptic case, it is sometimes possible to prove universal estimates of {\bf positive} solutions in spite
of the existence of arbitrary large semitrivial solutions. This is for instance the case for 
Dirichlet problems associated with system (\ref{c52}) when $r=1$, $m=2$, $\beta_{11}=\beta_{22}=0$, $\beta_{12}>0$ and $N\le 3$
(see \cite[Theorems 1.1 and 6.1]{MSS14}). 
However the following counter-example shows that even this fails in the parabolic case:
\smallskip

\noindent 
Let $r\geq 1$ and assume that $\beta_{11}=0$. Let  $A>0$ be fixed. For any $\varepsilon\in (0,A)$, we denote $U_\varepsilon$ the maximal classical solution of system (\ref{c51}) in $B_1\times [0,T)$ with initial and boundary 
 value $U_\varepsilon=(A,\varepsilon,...,\varepsilon)$. Then $u_{\varepsilon,i}\ge \varepsilon>0$   by the maximum principle.
Also, $U_0=(A,0,...,0)$ is a global solution of system (\ref{c51}) in $B_1\times [0,\infty)$ 
with initial and boundary value $(A,0,...,0)$. Since the nonlinearity is Lipschitz ($r\geq 1$),  it follows from the continuous dependence of solutions with respect to initial and boundary data that, for $\varepsilon\in (0,\varepsilon_0(A))$ with $\varepsilon_0(A) >0$ small enough,  the solution $U_\varepsilon$ exists in $B_1\times [0,1]$. Since $A>0$ is arbitrarily large, the universal estimate (\ref{universal}) thus fails also for positive solutions.
\end{remark}

For the next application of our Liouville-type theorems, we consider the initial-boundary 
 value problem:
\begin{align}\label{c5BVP}
\begin{cases}
\displaystyle \frac{\partial u_i}{\partial t}-\Delta u_i+\lambda_i u_i 
&=\displaystyle\sum_{j=1}^{m}\beta_{ij} u_i^ru_j^{r+1},\qquad  x\in \Omega,\quad t\in (0,T),\quad i=1,...,m,\\
\noalign{\vskip 1mm}
u_i&=0, \qquad  x\in \partial\Omega,\quad t\in (0,T),\quad i=1,...,m,\\
u_i(x,0)&=u_{0,i}, \qquad  x\in\Omega,\quad i=1,...,m,
\end{cases}
\end{align}
where  $U_0=(u_{0,i})$ are non-negative functions in $(C_0(\Omega))^m$, $\lambda_i\in\mathbb R$ are constants, and $\Omega$ is a regular domain (possibly unbounded) of ${\mathbb R}^N$. It is well known (see, e.g., Amann~\cite{Ama85}) that problem (\ref{c5BVP})  has a unique (mild) solution 
$$U(t)\in C\bigl([0, T); (C_0(\Omega))^m\bigr)$$ 
with maximal existence time   $T=T_{\text{max}}(U_0)$, and $U$ is a classical solution for $t\in(0,T)$.  
 We have the following universal estimates for global solutions, as well as universal initial and final time blow-up rates.
Again, this covers in particular the case $r=1$ in dimension~$N=3$.

\begin{proposition}\label{c5th4}   Let $r\geq 1$ and assume either $N\le 2$ or $N=3$ and $p=2r+1<p_B(3)=15/4$. Let $B$ satisfy (\ref{cooperativecond}).
Let $U$ be a  nonnegative solution of (\ref{c5BVP}) in $\Omega\times (0,T)$. 
\smallskip

(i) If $T<\infty$ then there holds
\begin{align*}
u_i(x,t)\leq C\left (1+ t^{-1/(p-1)}+ (T-t)^{-1/(p-1)} \right),\quad x\in \Omega,\; 0<t<T,\quad i=1,...,m,
\end{align*}
where $C=C(\Omega,  p, B)$. 
\smallskip

(ii) If $U$ is global then there holds
\begin{align*}
u_i(x,t)\leq C\left (1+ t^{-1/(p-1)}\right),\quad x\in \Omega,\; t>0,\quad i=1,...,m, 
\end{align*}
where $C=C(\Omega, p, B)$. 
\end{proposition}

The proof of Proposition~\ref{c5th4} is completely similar to that of \cite[Theorem 4.1]{PQS07b}, 
based on rescaling and doubling arguments, and Liouville-type theorems.
Namely, we use the Liouville-type Theorems~\ref{th1} in the whole space, and Theorem~\ref{Liouhalf} in a half-space.

\begin{remark} 
The upper restriction on $N$ in Proposition~\ref{c5th4} is a consequence of the condition $r\geq 1$,
required by Theorem~\ref{Liouhalf}.
However, if we consider only radial solutions in a symmetric domain $\Omega$ (i.e. the whole space ${\mathbb R}^N$, a ball, an annulus, or the complement of a ball), then Proposition~\ref{c5th4} is true under the weaker assumptions $r>0$, $N\ge 1$ and $p<p_S$ (see \cite[Section 2]{Pha14}).
\end{remark}

\section{Proof of Theorem~\ref{th1}}
For the sake of simplicity, we denote by $\int$ the integral  $ \int_{B_{1}}\int_{-1}^{1}dxdt$. 
The key step is the following Lemma.

\begin{lemma}\label{c5lem1}
 Assume that $p<p_B(N)$ and $B$ satisfies (\ref{cooperativecond}). Let $0\leq \varphi\in {\mathcal D}(B_1\times (-1, 1))$ and $U$ be a
 positive classical  solution of (\ref{c51}) in $B_1\times (-1, 1)$. Denote 
\begin{align*}
 &I_i=\int\varphi u_i^{-2}|\nabla u_i|^4, \qquad I=\sum_{i=1}^{m}I_i,\\  
&L_i=\int \varphi \Big(\sum_{j=1}^{m}\beta_{ij} u_i^ru_j^{r+1}\Big)^2, \qquad L=\sum_{i=1}^{m}L_i.
\end{align*}
Then there holds
\begin{align}
 I+L\leq &\ C\sum_{i=1}^m\int\varphi \Big(|\partial_tu_i| u_i^{-1}|\nabla u_i|^2+ |\partial_tu_i|^2\Big)+C\sum_{i=1}^m\int |\Delta\varphi||\nabla u_i|^2\notag\\
&+C\sum_{i=1}^m\int \Big(\sum_{j=1}^{m}\beta_{ij} u_i^ru_j^{r+1}+|\partial_tu_i|+u_i^{-1}|\nabla u_i|^2\Big)|\nabla\varphi.\nabla u_i|\notag\\
&+C\sum_{i,j=1}^m\int |\varphi_t|\beta_{ij}u_i^{r+1}u_j^{r+1},\label{c5aa3}
\end{align}
where $C=C(N,p, B)$. 
\end{lemma}

 \begin{proof} {\bf Step 1.} {\it Preparations.} 
 Denote
\begin{align*}
 &J_i:=\int \varphi u_i^{-1}|\nabla u_i|^2\Delta u_i, \qquad J=\sum_{i=1}^{m}J_i,\\  
 &K_i:=\int \varphi (\Delta u_i)^2,  \qquad K=\sum_{i=1}^{m}K_i.
\end{align*}
Applying \cite[Lemma 8.9]{QS07} with $q=0$, $-1\ne k<0$, we have
\begin{align*}
 -&\Big(\frac{N-1}{N}k+1\Big)kI_i+\frac{N+2}{N}kJ_i-\frac{N-1}{N}K_i\notag\\
 &\quad\leq \frac{1}{2}\int |\nabla u_i|^2\Delta\varphi+\int\bigl(\Delta u_i-ku_i^{-1}|\nabla u_i|^2\bigr)\nabla u_i.\nabla\varphi.
\end{align*}
(We stress that this is true if $u_i$ is any positive $C^{2,1}$ function, with no reference to the PDE system (\ref{c51}) at this point.) 
Therefore, 
\begin{align}\label{c5a111}
 &-\Big(\frac{N-1}{N}k+1\Big)kI+\frac{N+2}{N}kJ-\frac{N-1}{N}K\notag\\
 &\leq \sum_{i=1}^{m}\int \bigg(\frac{1}{2}|\nabla u_i|^2\Delta\varphi+\bigl(\Delta u_i-ku_i^{-1}|\nabla u_i|^2\bigr)\nabla u_i.\nabla\varphi\bigg).
\end{align}
\smallskip
{\bf Step 2.} {\it Estimate of $J$ and $K$.}
We claim that
\begin{align}\label{J}
-J\geq&\ \frac{1}{p}\sum_{i=1}^m\int\varphi \Big(\sum_{j=1}^m\beta_{ij} u_i^ru_j^{r+1}\Big)^2+\frac{1}{2p(r+1)}\int \varphi_t\sum_{i,j=1}^m\beta_{ij}u_i^{r+1}u_j^{r+1}\notag\\
&-\frac{1}{p}\int \sum_{i,j=1}^m\beta_{ij} u_i^ru_j^{r+1}|\nabla\varphi.\nabla u_i|-\int\varphi \sum_{i=1}^m|\partial_tu_i| u_i^{-1}|\nabla u_i|^2 
\end{align}
and
\begin{align}\label{K}
 K=\sum_{i=1}^{m}\int \varphi \Big(\sum_{j=1}^m\beta_{ij} u_i^ru_j^{r+1}\Big)^2+\sum_{i=1}^{m}\int \varphi |\partial_t u_i|^2
 +\frac{1}{r+1}\sum_{i,j=1}^m\int \varphi_t\beta_{ij}u_i^{r+1}u_j^{r+1}.
\end{align}

Let us first establish (\ref{J}). Using integration by parts, we have  
\begin{align}
\int\varphi|\nabla u_i|^2u_i^{r-1}u_j^{r+1}&=\int\varphi u_j^{r+1}\nabla u_i.\nabla\left(\frac{u_i^r}{r}\right)\notag\\
&=-\frac{1}{r}\int\varphi u_i^{r}u_j^{r+1}\Delta u_i-\frac{1}{r}\int u_i^ru_j^{r+1}\nabla \varphi.\nabla u_i\notag\\
&\;\quad-\frac{r+1}{r}\int \varphi u_i^{r}u_j^r \nabla u_i.\nabla u_j.    \label{p1}
\end{align}
Using the Young inequality $2u_i^{r}u_j^r \nabla u_i.\nabla u_j\leq |\nabla u_i|^2u_i^{r-1}u_j^{r+1}+|\nabla u_j|^2u_i^{r+1}u_j^{r-1}$, it follows from (\ref{p1}) that
\begin{align*}
\int\varphi|\nabla u_i|^2u_i^{r-1}u_j^{r+1}&\geq 
-\frac{1}{r}\int\varphi u_i^{r}u_j^{r+1}\Delta u_i-\frac{1}{r}\int u_i^ru_j^{r+1}\nabla \varphi.\nabla u_i\\
&\quad-\frac{r+1}{2r}\int \varphi \big(|\nabla u_i|^2u_i^{r-1}u_j^{r+1}+|\nabla u_j|^2u_i^{r+1}u_j^{r-1}\big).
\end{align*}
Consequently, 
\begin{align*}
\sum_{i,j=1}^m\beta_{ij}\int\varphi|\nabla u_i|^2u_i^{r-1}u_j^{r+1}&\geq 
-\frac{1}{r}\sum_{i,j=1}^m\int\varphi \beta_{ij}u_i^{r}u_j^{r+1}\Delta u_i-\frac{1}{r}\sum_{i,j=1}^m\int \beta_{ij}u_i^ru_j^{r+1}\nabla \varphi.\nabla u_i\\
&\quad-\frac{r+1}{2r}\sum_{i,j=1}^m\beta_{ij}\int \varphi \big(|\nabla u_i|^2u_i^{r-1}u_j^{r+1}+|\nabla u_j|^2u_i^{r+1}u_j^{r-1}\big),
\end{align*}
hence
\begin{align*}
\sum_{i,j=1}^m\beta_{ij}\int\varphi|\nabla u_i|^2u_i^{r-1}u_j^{r+1}&\geq 
-\frac{1}{r}\sum_{i,j=1}^m\int\varphi \beta_{ij}u_i^{r}u_j^{r+1}\Delta u_i-\frac{1}{r}\sum_{i,j=1}^m\int \beta_{ij}u_i^ru_j^{r+1}\nabla \varphi.\nabla u_i\\
&\quad-\frac{r+1}{r}\sum_{i,j=1}^m\beta_{ij}\int \varphi |\nabla u_i|^2u_i^{r-1}u_j^{r+1},
\end{align*}
owing to the symmetry property $\beta_{ij}=\beta_{ji}$.
Consequently, recalling $p=2r+1$, we obtain 
\begin{align}\label{p3}
\sum_{i,j=1}^m\beta_{ij}\int\varphi|\nabla u_i|^2u_i^{r-1}u_j^{r+1}\geq 
-\frac{1}{p}\int\sum_{i,j=1}^m\varphi\beta_{ij} u_i^{r}u_j^{r+1}\Delta u_i-\frac{1}{p}\int\sum_{i,j=1}^m\beta_{ij} u_i^ru_j^{r+1}\nabla \varphi.\nabla u_i.
\end{align}
Now substituting the PDE (\ref{c51}), written as
\begin{equation}\label{p3b}
-\Delta u_i=\sum_{j=1}^m\beta_{ij} u_i^ru_j^{r+1}-\partial_tu_i,
\end{equation}
in the definition of $J$, it follows from (\ref{p3}) that
\begin{align*}
-J=&\sum_{i,j=1}^m\int\varphi |\nabla u_i|^2\beta_{ij} u_i^{r-1}u_j^{r+1}-\sum_{i=1}^m\int\varphi u_i^{-1}|\nabla u_i|^2\partial_tu_i \\ 
\geq&\  \frac{1}{p}\sum_{i=1}^m\int\varphi \Big(\sum_{j=1}^m\beta_{ij} u_i^ru_j^{r+1}\Big)(-\Delta u_i)-\frac{1}{p}\sum_{i,j=1}^m\int \beta_{ij} u_i^ru_j^{r+1}|\nabla\varphi.\nabla u_i|\\
&-\sum_{i=1}^m\int\varphi |\partial_tu_i| u_i^{-1}|\nabla u_i|^2
\end{align*}
and then, that
\begin{align}\label{p4}
-J\geq&\ \frac{1}{p}\sum_{i=1}^m\int\varphi \Big(\sum_{j=1}^m\beta_{ij} u_i^ru_j^{r+1}\Big)^2
- \frac{1}{p}\int\varphi \sum_{i,j=1}^m\beta_{ij} u_i^ru_j^{r+1}\partial_tu_i\\
&-\frac{1}{p}\sum_{i,j=1}^m\int \beta_{ij} u_i^ru_j^{r+1}|\nabla\varphi.\nabla u_i|-\sum_{i=1}^m\int\varphi |\partial_tu_i| u_i^{-1}|\nabla u_i|^2.\notag
\end{align}
Next observe that, due to $\beta_{ij}=\beta_{ji}$, we have 
\begin{align}\label{IPP}
\partial_t\Big(\sum_{i,j=1}^m\beta_{ij} u_i^{r+1}u_j^{r+1}\Big)
&=(r+1)\sum_{i,j=1}^m\beta_{ij} u_i^ru_j^{r+1}\partial_tu_i
+(r+1)\sum_{i,j=1}^m\beta_{ij} u_i^{r+1}u_j^r\partial_tu_j \notag \\
&=2(r+1)\sum_{i,j=1}^m\beta_{ij} u_i^ru_j^{r+1}\partial_tu_i.
\end{align}
Combining (\ref{p4}) and (\ref{IPP}) and integrating by parts in $t$, we obtain
\begin{align*}
-J\geq&\ \frac{1}{p}\sum_{i=1}^m\int\varphi \Big(\sum_{j=1}^m\beta_{ij} u_i^ru_j^{r+1}\Big)^2+\frac{1}{2p(r+1)}\int \varphi_t\sum_{i,j=1}^m\beta_{ij}u_i^{r+1}u_j^{r+1}\notag\\ 
&-\frac{1}{p}\sum_{i,j=1}^m\int \beta_{ij} u_i^ru_j^{r+1}|\nabla\varphi.\nabla u_i|-\sum_{i=1}^m\int\varphi |\partial_tu_i| u_i^{-1}|\nabla u_i|^2 
\end{align*}
i.e., estimate (\ref{J}).

Next consider $K$. By substituting (\ref{p3b}) in the definition of $K$, we have
\begin{align*}
 K= \sum_{i=1}^{m}\int \varphi \Big(\sum_{j=1}^m\beta_{ij} u_i^ru_j^{r+1}\Big)^2+\sum_{i=1}^{m}\int \varphi |\partial_t u_i|^2
 -2\sum_{i,j=1}^m\int \varphi\beta_{ij}u_i^{r}u_j^{r+1}\partial_tu_i.
\end{align*}
Using (\ref{IPP}) again and integrating by parts in $t$, we obtain (\ref{K}).

\smallskip
{\bf Step 3.} {\it Conclusion.}
It follows from  (\ref{J}) and (\ref{K}) that, for any $-1\ne k<0$, 
\begin{align}
 \frac{N+2}{N}kJ-\frac{N-1}{N}K\geq &\ \Big(\frac{N+2}{pN}(-k)-\frac{N-1}{N}\Big)\sum_{i=1}^m\int\varphi\Big(\sum_{j=1}^m\beta_{ij} u_i^ru_j^{r+1}\Big)^2\notag\\ 
&-C\sum_{i,j=1}^m\int |\varphi_t|\beta_{ij}u_i^{r+1}u_j^{r+1}-C\sum_{i,j=1}^m\int \beta_{ij} u_i^ru_j^{r+1}|\nabla\varphi.\nabla u_i|.\notag\\
  &-C\sum_{i=1}^m\int\varphi |\partial_tu_i| u_i^{-1}|\nabla u_i|^2-C\sum_{i=1}^m\int\varphi |\partial_tu_i|^2 . \label{estimate2}
\end{align}
Since $p<p_B(N)$, we can take $k>-N/(N-1)$ close to $-N/(N-1)$ such that  
$$\Big(\frac{N-1}{N}k+1\Big)(-k)>0 \qquad \text{ and }\quad  \frac{N+2}{pN}(-k)-\frac{N-1}{N}>0.$$
(If $N=1$ then we take any $-1\ne k<0$.) The lemma follows from (\ref{c5a111}) and (\ref{estimate2}).
\end{proof}
\medskip

\begin{lemma}\label{c5lem2}
Assume that $p<p_B(N)$ and $B$ satisfies (\ref{cooperativecond}). Let $U$ be a positive classical solution of  (\ref{c51}) in $B_1\times (-1, 1)$. Then 
\begin{align} \label{c5bound}
 \int_{B_{1/2}}\int_{-1/2}^{1/2}\sum_{i=1}^{m} 
 \Big(\sum_{j=1}^{m}\beta_{ij} u_i^{r}u_j^{r+1}\Big)^2dxdt\leq C(N, p, B).
\end{align}
\end{lemma}
\begin{proof}

We  follow the argument as in the proof of \cite[Proposition 21.5]{QS07}.
One can choose the test-function $\varphi$ such that $\varphi=1$ in $B_{1/2}\times (-1/2,1/2)$, $0\leq \varphi\leq 1$ and 
\begin{equation} \label{PropTest}
 |\nabla\varphi|\leq C \varphi^{(3p+1)/4p}, \quad |\Delta\varphi|\leq C \varphi^{(p+1)/2p}, \quad |\varphi_t|\leq C\varphi^{(3p+1)/4p}\leq C\varphi^{(p+1)/2p}.
\end{equation}
Recall the notation $\int=\int_{B_{1}}\int_{-1}^{1}dxdt$. By the proof of \cite[Proposition 21.5]{QS07} (see formulae (21.10) and (21.11)),
for any $\varepsilon>0$ and any positive function $u\in C^{2,1}(B_1\times (-1,1))$ , we have 
$$
\int|\nabla u|^2(|\Delta \varphi|+\varphi^{-1}|\nabla\varphi|^2+|\varphi_t|)
\leq \varepsilon\int\varphi \bigl(u^{-2}|\nabla u|^4+u^{2p}\bigr)+C(\varepsilon)
$$
and
\begin{align*}
&\int\Big(\varphi|u_t|u^{-1}|\nabla u|^2+\big(u^p+|u_t|+u^{-1}|\nabla u|^2\big)|\nabla u.\nabla \varphi|+u^{p+1}|\varphi_t|\Big)\\
&\qquad\qquad \leq \varepsilon\int\varphi \bigl(u^{-2}|\nabla u|^4+u^{2p}\bigr)+C(\varepsilon)\Big(1+\int\varphi|u_t|^2\Big).
\end{align*}
Set $\beta=\displaystyle\min_{1\le i\le m}\beta_{ii}>0$. Using 
$$
I+L\ge \sum_{i=1}^{m}\int\varphi u_i^{-2}|\nabla u_i|^4+\beta^2\sum_{i=1}^{m} \int\varphi u_i^{2p}
$$
and
$$u_i^ru_j^{r+1}\le u_i^p+u_j^p,\qquad u_i^{r+1}u_j^{r+1}\le u_i^{p+1}+u_j^{p+1},$$
it follows that
\begin{align}\label{mix}
 &\int|\nabla u_i|^2(|\Delta \varphi|+\varphi^{-1}|\nabla\varphi|^2+|\varphi_t|)\leq \varepsilon(I+L)+C(\varepsilon),\notag\\
&\int\Big(\varphi|\partial_tu_i|u_i^{-1}|\nabla u_i|^2+\big(\sum_{j=1}^{m}\beta_{ij} u_i^ru_j^{r+1}+|\partial_tu_i|+u_i^{-1}|\nabla u_i|^2\big)|\nabla u_i.\nabla \varphi|\Big)\notag\\
&\qquad\qquad \leq \varepsilon(I+L)+C(\varepsilon)\Big(1+\int\varphi|\partial_tu_i|^2\Big),\\
&\int\Big(\sum_{j=1}^{m}\beta_{ij} u_i^{r+1}u_j^{r+1}|\varphi_t|\Big) \leq \varepsilon(I+L)+C(\varepsilon)\Big(1+\int\varphi|\partial_tu_i|^2\Big).\notag
\end{align}
On the other hand, using (\ref{IPP}) and integrating by parts in space and in time, we have 
\begin{align*}
 &\sum_{i=1}^{m}\int \varphi|\partial_tu_i|^2
 =\sum_{i=1}^{m}\int \varphi (\partial_tu_i)\Big(\Delta u_i+\sum_{j=1}^{m}\beta_{ij} u_i^ru_j^{r+1}\Big)\\
&=-\sum_{i=1}^{m}\int\varphi \partial_t\Big(\frac{|\nabla u_i|^2}{2}\Big)
+\frac{1}{2(r+1)}\int\varphi \partial_t\Big(\sum_{i,j=1}^{m}\beta_{ij} u_i^{r+1}u_j^{r+1}\Big)-\sum_{i=1}^{m}\int (\partial_tu_i)\nabla\varphi.\nabla u_i\\
&=\int\varphi_t\sum_{i=1}^{m}\Big(\frac{|\nabla u_i|^2}{2}\Big)
-\frac{1}{2(r+1)}\int\varphi_t\Big(\sum_{i,j=1}^{m}\beta_{ij} u_i^{r+1}u_j^{r+1}\Big)-\sum_{i=1}^{m}\int (\partial_tu_i)\nabla\varphi.\nabla u_i\\
&\leq C\int|\varphi_t|\Big(\sum_{i=1}^{m}|\nabla u_i|^2+\sum_{i,j=1}^{m}\beta_{ij}u_i^{r+1}u_j^{r+1}\Big)\\
&\qquad+ \frac{1}{2}\sum_{i=1}^{m}\int\varphi|\partial_tu_i|^2+\frac{1}{2}\sum_{i=1}^{m} \int \varphi^{-1}|\nabla\varphi|^2|\nabla u_i|^2.
\end{align*}
(Note that, since $u_i$ is positive, it is smooth enough so that the above calculations are justified.) 
Therefore,
\begin{align}\label{timederivative}
 \sum_{i=1}^{m}\int \varphi|\partial_tu_i|^2&\leq C\sum_{i=1}^{m}\int|\varphi_t| \Big(|\nabla u_i|^2+\sum_{j=1}^{m}\beta_{ij}u_i^{r+1}u_j^{r+1}\Big)+ \sum_{i=1}^{m}\int \varphi^{-1}|\nabla\varphi|^2|\nabla u_i|^2\notag\\
 &\leq C\sum_{i,j=1}^{m}\int|\varphi_t| \beta_{ij}u_i^{r+1}u_j^{r+1}+ C(N,p, B)\varepsilon(I+L)+C(\varepsilon).
\end{align}
By (\ref{PropTest}) and Young 's inequality,
\begin{align*}
 \int |\varphi_t|u_i^{r+1}u_j^{r+1}&\leq 2\varepsilon\int\varphi(u_iu_j)^p+ C(\varepsilon)\int\varphi^{-(p+1)/(p-1)}|\varphi_t|^{2p/(p-1)}\\
 &\leq \varepsilon\int\varphi \Big(u_i^{2p}+ u_j^{2p}\Big)+ C(\varepsilon)\int\varphi^{-(p+1)/(p-1)}|\varphi_t|^{2p/(p-1)}\\
 &\leq \frac{\varepsilon}{\beta^2_{ii}}L_i+ \frac{\varepsilon}{\beta^2_{jj}}L_j+ C(\varepsilon). 
\end{align*}
Hence, 
\begin{align}
 \int |\varphi_t|\beta_{ij}u_i^{r+1}u_j^{r+1}&\leq \beta_{ij}\Big(\frac{\varepsilon}{\beta^2_{ii}}+\frac{\varepsilon}{\beta^2_{jj}}\Big)(I+L)+ C(\varepsilon).\label{time2}
\end{align}
Combing (\ref{timederivative}) and (\ref{time2}), we obtain
\begin{align}\label{time3}
\sum_{i=1}^{m}\int \varphi|\partial_tu_i|^2\leq C(N,p,B)\varepsilon(I+L)+ C(\varepsilon).
\end{align}
Therefore, it follows from (\ref{c5aa3}), (\ref{mix}) and (\ref{time3}) that
\begin{align*}
 I+L \leq C(\varepsilon)+ C(N,p,B) \varepsilon (I+L).
\end{align*}
By choosing $\varepsilon$ sufficiently small, we obtain $I,L \leq C$ and the Lemma follows.
\end{proof}

\medskip

{\it Proof of Theorem~\ref{th1}.}  
We first consider the case of bounded solutions.
Assume for contradiction that 
$U=(u_i)$ is a  nontrivial, bounded, nonnegative solution of (\ref{c51}) in ${\mathbb R}^N\times \mathbb R$.  
For each $i$, since the component $u_i$ is a supersolution of the heat equation, it
follows from the strong maximum principle that either $u_i$ is positive in ${\mathbb R}^N\times \mathbb R$,
or there exists $t_0\in{\mathbb R}$ such that $u_i=0$ in ${\mathbb R}^N\times (-\infty,t_0]$. 
In the latter case, since $U$ is bounded, we have $\partial_t u_i-\Delta u_i\le Cu_i$ for some constant $C>0$ and the maximum principle
then guarantees that $u_i=0$ in ${\mathbb R}^N\times (t_0,\infty)$, hence $u_i\equiv 0$.
By removing all the components which are identically zero and relabeling, we may assume without loss of generality that $u_i>0$ for $i=1,...,m$
(with $m\ge 2$, since nonexistence in the scalar case is already known by \cite{BV98}, see also \cite[Theorem~26.8]{QS07}). 

Now, for any $R>0$, we rescale
$$v_i(x,t)=R^{2/(p-1)}u_i(R x, R^2 t).$$
Then $V=(v_i)$ is also  a solution of (\ref{c51}). Since all $v_i$ are positive, we may apply Lemma~\ref{c5lem2} to deduce
\begin{align*} 
\int_{|y|<R/2}&\int_{|s|<R^2/2}\sum_{i=1}^{m}
\Big(\sum_{j=1}^{m}\beta_{ij} u_i^{r}u_j^{r+1}\Big)^2(y,s)dyds\\
&=R^{N+2-4p/(p-1)} \int_{|x|<1/2}\int_{|t|<1/2}\sum_{i=1}^{m}
\Big(\sum_{j=1}^{m}\beta_{ij} v_i^{r}v_j^{r+1}\Big)^2(x,t)dxdt\\
&\leq CR^{N+2-4p/(p-1)}.
\end{align*}
Since $p<p_B\le p_S$, by letting $R\to \infty$ we get $\sum_{i=1}^{m}
\bigl(\sum_{j=1}^{m}\beta_{ij} u_i^{r}u_j^{r+1}\bigr)^2\equiv 0$. 
Since $\beta_{ii}$ and $u_i$ are positive, this is a contradiction. 
This proves Theorem~\ref{th1} in the case of bounded solutions.

Finally, to treat the general case, we recall that the Liouville-type property of Theorem~\ref{th1} for bounded solutions
is sufficient for the proof of  Proposition~\ref{c5th5} (see the paragraph after Proposition~\ref{c5th5}).
But after a time shift, formula (\ref{universal2}) in Proposition~\ref{c5th5}
guarantees that any solution of (\ref{c51}) in ${\mathbb R}^N\times(-T,T)$ has to satisfy
$u(x,t)\le CT^{-1/(p-1)}$ in ${\mathbb R}^N\times(-T/2,T/2)$.
The conclusion then follows by letting $T\to \infty$.
\qed

\section{Proof of Theorems~\ref{Liouhalf} and \ref{general}}

\begin{proof}[Proof of Theorem~\ref{general}] It is done in three steps.
\smallskip

 {\bf Step 1.} {\it Notation.}
For $\lambda>0$, let 
\begin{align*}
 \mathbb T_\lambda=\{x\in {\mathbb R}^N:0<x_1<\lambda\}.
\end{align*}
Denote  
$$V_\lambda u_i(x,t):=u_i(2\lambda-x_1, x',t)-u_i(x_1,x',t),$$
 where $x'=(x_2,...,x_N)$. Let $v_i=V_\lambda u_i$, then $V=(v_i)$ satisfies
\begin{align}\label{moving}
\begin{cases}
\partial_t v_i-\Delta v_i =\displaystyle\sum_{j=1}^{m}c^\lambda_{ij}v_j, \quad (x,t)\in {\mathbb T_\lambda}\times \mathbb R, \quad i=1,2,...,m.\\
v_i=0, \quad x_1=\lambda,\; x'\in {\mathbb R}^{N-1}, \; t\in {\mathbb R},\;  i=1,...,m,\\
\noalign{\vskip 1mm}
v_i>0, \quad x_1=0,\; x'\in {\mathbb R}^{N-1}, \; t\in {\mathbb R},\;  i=1,...,m,
\end{cases}
\end{align}
where 
\begin{align}\label{moving2}
 c^\lambda_{ij}&=\int_{0}^{1}\frac{\partial f_i}{\partial u_j}(U+sV)ds.
\end{align}
We shall show that,  for any $\lambda>0$,
\begin{align}\label{S}
V\geq 0 \ \ \text{ for all } (x,t)\in {\mathbb T_\lambda}\times {\mathbb R}. \tag{S}
\end{align} 

For any positive $q,\lambda$ satisfying $\lambda\sqrt{q}<\pi$, we define the function
\begin{equation}\label{defh}
h(x)=\sin\Bigl[\frac12\Bigl(\pi+\sqrt{q} (2x_1-\lambda)\Bigr)\Bigr],
\end{equation}
which satisfies
\begin{align}\label{dancer}
\begin{cases} 
 -\Delta h=qh, \quad &x\in {\mathbb T}_\lambda,\\
 h(x)\ge \eta>0, \quad &x\in \overline{\mathbb T}_\lambda, \\
 |\nabla h|\le \sqrt{q}, \quad &x\in \overline{\mathbb T}_\lambda,
 \end{cases}
\end{align}
where $\eta=\sin\bigl[\frac12(\pi-\lambda\sqrt{q})\bigl]$.\footnote{ \,We note that this choice of $h$ is simpler than that in \cite{Dan92, PQS07b}, 
owing to the different form of maximum principle used in the subsequent steps
(which does not require $W\to 0$ at space infinity in (\ref{w}) or (\ref{w2})).}

\smallskip
 {\bf Step 2.} {\it Proof of (S) for small $\lambda$.} 
 Due to the boundedness of $U$, the coefficients $c^\lambda_{ij}$ are bounded (uniformly in $\lambda$). So we may fix $\gamma, q>0$ such that
\begin{align}\label{q}
q\ge \gamma+ \max_{i}\sup_{x\in {\mathbb R}^N_{+},\, t\in {\mathbb R},\, \lambda>0}\ \sum_{j=1}^{m}c^\lambda_{ij}(x,t). 
\end{align}

For any $\lambda>0$ small such that $\lambda^{-2}\pi^2>q$, let $h$ be  given by (\ref{defh}) and denote $w_i:=e^{\gamma t}v_i/h$.
Then $W=(w_i)$ satisfies 
\begin{align}\label{w}
\begin{cases}
\partial_t w_i-\Delta w_i -\frac{2\nabla h}{h}.\nabla w_i=(\gamma+c^{\lambda}_{ii}-q)w_i+\displaystyle\sum_{j=1\atop j\ne i}^mc^\lambda_{ij}w_j, 
\quad (x,t)\in {\mathbb T_\lambda}\times \mathbb R, \; i=1,...,m,\\
w_i\geq 0, \quad x\in \partial{\mathbb T_\lambda},\; t\in {\mathbb R},\;  i=1,...,m.
\end{cases}
\end{align}
 By assumption $(H_2)$, we see that $c^\lambda_{ij}$ is nonnegative for all $i\neq j$, so that the system (\ref{moving}) is cooperative.  
Moreover, by the definition of $q$ in (\ref{q}), we have
$$\bigl(\gamma+c^{\lambda}_{ii}-q\bigr)+\sum_{j\ne i}c^\lambda_{ij}\leq 0, \text { for all } i=1,...,m.$$
 We may thus apply the maximum principle for cooperative parabolic systems (see~Proposition~\ref{maxprinc}(ii) in Appendix).
The latter,  applied to $-W$ on ${\mathbb T_\lambda}\times (t_0,t)$ for any $t_0<t$, guarantees that 
\begin{align}\label{maxprincapp1}
\max_{i}\sup_{x\in {\mathbb T_\lambda}}w_i^{-}(x,t)\leq \max_{i}\sup_{x\in {\mathbb T_\lambda}}w_i^{-}(x,t_0),
\end{align}
where $z^{-}:=-\min(z,0)$. 
Consequently, 
\begin{align}\label{maxprincapp2}
\max_{i}\sup_{x\in {\mathbb T_\lambda}}\frac{v_i^{-}(x,t)}{h(x)}\leq e^{-\gamma(t-t_0)}\max_{i}\sup_{x\in {\mathbb T_\lambda}}\frac{v_i^{-}(x,t_0)}{h(x)}
 \leq e^{-\gamma(t-t_0)}\frac{\|U\|_\infty}{\eta}.
\end{align}
Letting $t_0\to -\infty$, we obtain $V\geq 0$ in ${\mathbb T_\lambda}\times \mathbb R$. 
Therefore, (S) holds for any $\lambda>0$ sufficiently small.

\smallskip
 {\bf Step 3.} {\it Proof of (S) for large $\lambda$.} 
Next, we denote
\begin{align}\label{lam}
\lambda_0=\sup\{\mu>0: \text{(S) holds for all } \lambda \in (0,\mu) \}.
\end{align}
The previous argument shows that $\lambda_0>0$, and we shall prove by contradiction that $\lambda_0=\infty$.
Assume $\lambda_0<\infty$. Then there is a sequence $\lambda_k\geq \lambda_0$ such that $\lambda_k\to \lambda_0$ and the set
\begin{align*}
Z_k=\{(x,t)\in {\mathbb T_{\lambda_k}}\times {\mathbb R}: \text{there exists } i \text{ such that } V_{\lambda_k}u_i(x,t)<0\} 
\end{align*}
is nonempty. Set
\begin{align*}
 m_k:=\sup_{(x,t)\in Z_k} \ \max\Big\{\max_{i}u_i(x,t),\ \max_{i}u_i(2\lambda_k-x_1,x',t)\Big\}.
\end{align*} 
We have the following two possibilities:
 
\smallskip
Case 1: $m_k\geq \varepsilon_0$ for some $\varepsilon_0>0$,  up to a subsequence;

\smallskip
Case 2: $m_k\to 0$.
 
\smallskip
First consider Case 1. Then there exist  sequences $x_1^k\in (0,\lambda_k)$, $z^k\in {\mathbb R}^{N-1}$, $t^k\in {\mathbb R}$ such that 
 $$(x_1^k, z^k,t^k)\in Z_k \  \text{ and } \max\Big\{\max_{i}u_i(x_1^k, z^k,t^k),\ \max_{i}u_i(2\lambda_k-x_1^k, z^k,t^k)\Big\} \geq \varepsilon_0 \;
 \text{ for any } k.$$
By extracting a subsequence of $k$, we may assume that there exists $i_0\in \{1,...,m\}$  such that,
for any $k$, 
\begin{equation}\label{Vlambdak0}
V_{\lambda_k}u_{i_0}(x_1^k,z^k, t^k)<0
\end{equation}
and
\begin{equation}\label{Vlambdak}
\max\Big\{\max_{i}u_i(x_1^k, z^k,t^k),\ \max_{i}u_i(2\lambda_k-x_1^k, z^k,t^k)\Big\} \geq \varepsilon_0.
\end{equation}
We may also assume that $x_1^k\to a$  for some $a\in [0,\lambda_0]$. Let
\begin{align*}
u_i^k(x,t)=u_i(x_1,x'+z^k, t+t^k), \qquad x=(x_1, x')\in {\mathbb R}^N_{+}, \quad t\in {\mathbb R}, \quad i=1,...,m.
\end{align*}
Since the sequence $U^k=(u_i^k)$ is uniformly bounded, using standard parabolic estimates, 
it follows that $U^k$ converges (up to a subsequence) in $C^{2,1}_{loc}(\overline{{\mathbb R}^N_{+}}
\times {\mathbb R})$ to a nonnegative solution $\tilde{U}=(\tilde{u}_i)$ of (\ref{coop}).

The definition of $\lambda_0$ implies that $V_{\lambda_0}u^k_{i}\geq 0$ in ${\mathbb T_{\lambda_0}}\times {\mathbb R}$ for any $i=1,...,m$. Hence, $V_{\lambda_0}\tilde{u}_{i}\geq 0$ in ${\mathbb T_{\lambda_0}}\times {\mathbb R}$ for any $i=1,...,m$.
Let $\tilde{v}_i:=V_{\lambda_0}\tilde{u}_i$. Then $(\tilde{v}_i)$ is a nonnegative solution of the system 
\begin{align*}
\begin{cases}
\partial_t \tilde{v}_i-\Delta \tilde{v}_i = \displaystyle\sum_{j=1}^{m}c^{\lambda_0}_{ij}\tilde{v}_j, \quad (x,t)\in {\mathbb T_{\lambda_0}}\times \mathbb R, \quad i=1,2,...,m,\\
\tilde{v}_i=0, \quad x_1=\lambda_0,\; x'\in {\mathbb R}^{N-1}, \; t\in {\mathbb R},\;  i=1,...,m,\\
\noalign{\vskip 1mm}
\tilde{v}_i\geq 0, \quad x_1=0,\; x'\in {\mathbb R}^{N-1}, \; t\in {\mathbb R},\;  i=1,...,m.
\end{cases}
\end{align*}
 From (\ref{Vlambdak0}) and (\ref{Vlambdak}), we deduce that 
 \begin{equation}\label{Vlambdak1}
V_{\lambda_0}\tilde{u}_{i_0}(a,0,0)\leq 0
\end{equation}
and
\begin{equation}\label{Vlambdak2}
\max_{i}\tilde{u}_{i}(a,0,0)\geq \varepsilon_0\ \text{ or } \ \max_{i}\tilde{u}_{i}(2\lambda_0-a,0,0)\geq \varepsilon_0.
\end{equation}
Owing to assumption $(H_4)$,  (\ref{Vlambdak2}) guarantees that
\begin{align*}
\tilde{u}_{i_0}(x,t)>0, \text{ for all } (x,t)\in  {\mathbb R}^N_{+}\times \mathbb R,
\end{align*}
hence $V_{\lambda_0}\tilde{u}_{i_0}(0,x',t)>0$ for any $x',t$.  
Since $z:=V_{\lambda_0}\tilde{u}_{i_0}$ satisfies
\begin{align*}
\begin{cases}
\partial_t z-\Delta z \ge c^{\lambda_0}_{{i_0}{i_0}}z, \quad (x,t)\in {\mathbb T_{\lambda_0}}\times \mathbb R,\\
z=0, \quad x_1=\lambda_0,\; x'\in {\mathbb R}^{N-1}, \; t\in {\mathbb R},\\
\noalign{\vskip 1mm}
 z > 0, \quad x_1=0,\; x'\in {\mathbb R}^{N-1}, \; t\in {\mathbb R},
\end{cases}
\end{align*}
we deduce from the (scalar) maximum principle that $V_{\lambda_0}\tilde{u}_{i_0}>0$ in ${\mathbb T_{\lambda_0}}\times {\mathbb R}$. 
It thus follows from (\ref{Vlambdak1}) that $a=\lambda_0$.
 By the Hopf boundary principle,
$$2\partial_{x_1}\tilde{u}_{i_0}(\lambda_0,0,0)=-\partial_{x_1}V_{\lambda_0}\tilde{u}_{i_0}(x_1,0,0)|_{x_1=\lambda_0}
 =-\partial_{x_1}z(\lambda_0,0,0) >0.$$
Consequently, $\partial_{x_1}\tilde{u}_{i_0}(x_1,0,0)$ is bounded below by a positive constant on an interval around $\lambda_0$ and this remains valid if $\tilde{u}_{i_0}$ is replaced by $u^k_{i_0}\approx \tilde{u}_{i_0}$. That is, there is $\delta >0$ such that
\begin{align}\label{par}
\partial_{x_1}u_{i_0}(x_1,z^k,t^k)=\partial_{x_1}u^k_{i_0}(x_1,0,0)>0, \; x_1\in [\lambda_0-\delta, \lambda_0+\delta]
\end{align}
 for $k$ sufficiently large. However, since $2\lambda_k-x_1^k>x_1^k$ both belong to  $[\lambda_0-\delta,\lambda_0+\delta]$ for large $k$, (\ref{par}) contradicts $V_{\lambda_k}u_{i_0}(x_1^k,z^k,t^k)<0$.
 
\smallskip
We next turn to Case 2. We consider the system (\ref{moving}) for large $k$. We are going to apply the maximum principle again,
 this time taking advantage of the fact that the parabolic inequalities need be satisfied only on the 
possible positivity set $Z_k$ (see~Proposition~\ref{maxprinc}(ii) in Appendix).
 Recalling $V=(V_{\lambda_k}u_i)$, it follows from $m_k\to 0$ that
$$\sup_{(x,t)\in Z_{k}} \bigl(|U(x,t)|+|V(x,t)|\bigr) \to 0 \; \text { as } k\to \infty.$$ 
In view of assumption $(H_3)$ and the definition (\ref{moving2}),  we deduce that
$$\limsup_{k\to \infty} \tilde{q}_k\  \le\ 0,\quad\hbox{ where }\tilde{q}_k:= \max_{i}\;\sup_{(x,t)\in Z_k}\sum_{j=1}^m c^{\lambda_k}_{ij}(x,t). $$ 
Fix a large $k$ and contants $q,\gamma>0$ such that $q=\tilde{q}_k+\gamma<\lambda_k^{-2}\pi^2 (\approx \lambda_0^{-2}\pi^2)$.
 Like in Step~2, we consider $w_i:=e^{\gamma t}v_i/h$ with $h$ given by (\ref{defh}) for $\lambda=\lambda_k$.
Let 
$$D_i:=\{(x,t)\in {\mathbb T_{\lambda_k}}\times \mathbb R;\ w_i(x,t) <0\}$$
 and note that $D_i\subset Z_k$. Then $(w_i)$ satisfies
\begin{align}\label{w2}
\begin{cases}
\partial_t w_i-\Delta w_i -\frac{2\nabla h}{h}.\nabla w_i=(\gamma+c^{\lambda_k}_{ii}-q)w_i+\displaystyle\sum_{j=1\atop j\ne i}^mc^{\lambda_k}_{ij}w_j, 
\quad (x,t)\in D_i, \ \ i=1,...,m,\\
w_i\geq 0, \quad (x,t)\in \partial{\mathbb T_{\lambda_k}}\times\mathbb R,\ \  i=1,...,m.
\end{cases}
\end{align}
By the maximum principle in Proposition~\ref{maxprinc}(ii), applied to $-W$, 
we obtain again (\ref{maxprincapp1})-(\ref{maxprincapp2}), and conclude that
$V  \geq 0$ in $\mathbb T_{\lambda_k}\times {\mathbb R}$. 
But this is a contradiction with the nonemptyness of $Z_k$. 

We have  thus reached a contradiction in both cases, which proves that $\lambda_0=\infty$ i.e., (S)~holds for any $\lambda>0$.  The Hopf boundary principle then gives
$$2\partial_{x_1}u_i(x,t)|_{x_1=\lambda}=-\partial_{x_1}V_{\lambda_0}u_i(x,t)|_{x_1=\lambda}>0$$
for any $\lambda>0$ and any $i$.  The theorem is proved.
\end{proof}

{\it Proof of Theorem~\ref{Liouhalf}.} It is done by induction on $m$.  
For $m=1$, it is reduced to the Liouville-type theorem for the equation $u_t-\Delta u=u^p$ in the half-space ${\mathbb R}^N_+\times \mathbb R$. 
This was proved in \cite[Theorem 2.1]{PQS07b}  for $N\le 2$, or $N\ge 3$ and $p<p_B(N-1)$. The result of Quittner~\cite{Qui14}, in conjunction with \cite[Theorem 2.4]{PQS07b}, guarantees the case $N=3$.  

Now fix $m\ge 2$ and assume that the theorem  holds for $m-1$. We shall prove  by contradiction that the theorem then  holds for $m$. Thus suppose that $U=(u_1,\dots,u_m)$   is a nontrivial, nonnegative, bounded classical solution of (\ref{halfspace}). 

 First we claim that 
\begin{align}\label{claimuipos1} 
\hbox{ any component $u_i$ is either identically zero or positive in ${\mathbb R}^N_+\times \mathbb R$.} 
\end{align} 
 Indeed
since $u_i$ is a supersolution to  the heat equation, it follows from  
the strong maximum principle that either $u_i$ is positive in ${\mathbb R}^N_{+}\times \mathbb R$,
or there exists $t_0\in{\mathbb R}$ such that $u_i=0$ in ${\mathbb R}^N_{+}\times (-\infty,t_0]$. 
In the latter case, since $U$ is bounded,  we have $\partial_t u_i-\Delta u_i\le Cu_i$ for some constant $C>0$ 
and the maximum principle then guarantees that $u_i=0$ in ${\mathbb R}^N_{+}\times (t_0,\infty)$, hence $u_i\equiv 0$.

 Next we claim that actually 
\begin{align}\label{uipos}
u_i>0 \quad\text{ for all }  i=1,...,m.
\end{align} 
Indeed, otherwise, by (\ref{claimuipos1}), there is a component $u_{i_0}$ which is identically zero. By removing this component, 
 we then obtain a nontrivial, nonnegative, bounded classical solution of problem (\ref{halfspace}) for $(m-1)$ components.
But this contradicts the induction assumption. 

We note that, owing to (\ref{cooperativecond}) and $r\geq1$, the system  satisfies all the conditions in Theorem~\ref{general}.
 In particular, assumption $(H_4)$ follows from (\ref{uipos}). Therefore, $\partial_{x_1}u_i(x,t)>0$ for all $(x,t)\in {{\mathbb R}^N_{+}}\times {\mathbb R}$ 
and for all $i=1,...,m$. 
Now, let 
$$u_{i,l}(x_1,x',t)=u_i(x_1+l, x',t), \quad (x_1, x',t)\in (-l,\infty)\times{\mathbb R}^{N-1}\times {\mathbb R}, \quad i=1,...,m.$$
From the boundedness of $u_i$ and parabolic estimates, letting $l\to \infty$ upon a subsequence, we can assume that $u_{i,l}$ converges uniformly on each compact set to $u_{i,\infty}$, where $(u_{i,\infty})$ is a bounded, nonnegative classical solution of problem~(\ref{c51}) in ${\mathbb R}^N\times {\mathbb R}$.  
The monotonicity of $u_i$ implies that $u_{i,\infty}$ is positive and independent of $x_1$. 
 We thus obtain a bounded, positive classical solution of problem~(\ref{c51}) in ${\mathbb R}^{N-1}\times {\mathbb R}$.
This contradicts Theorem~\ref{th1} and Remark~\ref{rm1}(b)
(here we assumed $N\ge 2$;  the case $N=1$ reduces to an ODE system for which nonexistence is obvious). 
\qed

\section{Appendix}
We give the following version of the maximum principle for cooperative systems, which is suitable to our needs.
Related results are given in \cite[Section 3.8]{PW67} or \cite[Theorem 3.2]{FP09}, but do not quite satisfy 
our requirements (unbounded domain, parabolic inequalities assumed on the positivity set only).
Here, for given vector $W:=(w_i)_{1\le i\le m}$ and real number $M$, the inequality $W\le M$ is understood as $w_i\le M$ for all $i=1,\dots,m$.

\begin{proposition}\label{maxprinc}
Let $m\ge 2$, $T>0$, let $\Omega$ be an arbitrary domain of ${\mathbb R}^N$
 (bounded or unbounded, not necessarily smooth).
 We denote $Q_T=\Omega\times(0,T)$ and $\partial_P Q_T=(\overline\Omega\times\{0\})\cup(\partial\Omega\times (0,T))$ its parabolic boundary.
Let $W=(w_i)\in C(\overline\Omega\times[0,T);{\mathbb R}^m)\cap C^{2,1}(Q_T;{\mathbb R}^m)$ and denote
$$D_i=\bigl\{(x,t)\in Q_T:\ w_i(x,t)>0\bigr\}.$$
Assume that $W$ is a bounded, classical solution of the system
\begin{equation}\label{ineqparab}
\partial_t w_i-\Delta w_i-K|\nabla w_i|\le  \sum_{j=1}^m c_{ij}(x,t)w_j\quad\hbox{ in $D_i$,\quad $i=1,\dots,m$},
\end{equation}
where $K>0$ is a constant and the coefficients $c_{ij}$ are measurable, bounded  and satisfy 
\begin{equation}\label{condoffdiag}
c_{ij}\ge 0 \quad\hbox{for all $i\neq j$.}
\end{equation}

(i) If $W\le 0$ on $\partial_PQ_T$, then $W\le 0$ in $Q_T$.
\smallskip

(ii) Let $M>0$ and assume in addition that 
\begin{equation}\label{condsum}
\sum_{j=1}^m c_{ij}\le 0,\quad i=1,\dots,m.
\end{equation}
If $W\le M$ on $\partial_PQ_T$, then $W\le M$ in $Q_T$.
\end{proposition}

{\it Proof.} (i)
It follows by the Stampacchia method, e.g. along the lines of \cite[Proposition 52.21]{QS07} and \cite[Remark 52.11(a)]{QS07}. 
We give the proof for the convenience of the reader and for completeness.

First consider the case when $\Omega$ is bounded. Let $i\in\{1,\dots,m\}$. 
By (\ref{ineqparab}), we have
$$\bigl[\partial_t w_i-\Delta w_i-K|\nabla w_i|\bigr](w_i)_+\le \sum_{j=1}^m c_{ij}(x,t)(w_j)(w_i)_+
\quad\hbox{in $Q_T$.}$$
For $t\in (0,T)$, since $(w_i)_+(\cdot,t)\in H^1_0(\Omega)$ by our assumption, we may integrate by parts, to obtain
\begin{align*}\frac12 \frac{d}{dt}\int_\Omega (w_i)_+^2
&=\int_\Omega (\partial_t w_i) (w_i)_+ \\
&\le -\int_\Omega \nabla w_i\cdot\nabla (w_i)_+ +K\int_\Omega |\nabla w_i| (w_i)_+
+ \sum_{j=1}^m \int_\Omega c_{ij}(w_j)(w_i)_+\\
&\le -\int_\Omega |\nabla (w_i)_+|^2 +\int_\Omega |\nabla (w_i)_+|^2 +\frac{K^2}{4}\int_\Omega (w_i)_+^2
+ \sum_{j=1}^m \int_\Omega c_{ij}(w_j)(w_i)_+.
\end{align*}
By assumption (\ref{condoffdiag}), it follows that
\begin{align*}
\frac12 \frac{d}{dt}\int_\Omega (w_i)_+^2
&\le\frac{K^2}{4}\int_\Omega (w_i)_+^2 + \int_\Omega c_{ii}(w_i)_+^2
+ \sum_{{j=1}\atop j\ne i}^m \int_\Omega c_{ij}(w_j)_+(w_i)_+ \\
&\le\Bigl(\frac{K^2}{4}+\|c_{ii}\|_\infty\Bigr)\int_\Omega (w_i)_+^2
+ \frac12\sum_{{j=1}\atop j\ne i}^m \int_\Omega c_{ij}\bigl((w_j)_+^2+(w_i)_+^2\bigr).
\end{align*}
Adding up for $i=1,\dots,m$, we get
$$\frac{d}{dt}\sum_{j=1}^m \int_\Omega (w_i)_+^2\le L\sum_{j=1}^m  \int_\Omega (w_i)_+^2$$
for some constant $L>0$.
Since $W\le 0$ at $t=0$, it follows by integration that $\sum_{j=1}^m \int_\Omega (w_i)_+^2\le 0$, hence 
$W\le 0$ in $Q_T$.
\smallskip

Now, in the case of an unbounded domain, we fix $\eps>0$ and consider the modified functions
$$\tilde w_i=e^{-\lambda t}w_i-\eps\psi, \qquad \psi=(N+K)t +(1+|x|^2)^{1/2}>0$$
with $\lambda>0$ to be chosen. We also set
$$\tilde D_i=\bigl\{(x,t)\in Q_T;\ \tilde w_i(x,t)>0\bigr\}.$$
Since $\psi_t-\Delta\psi-K|\nabla\psi|\ge 0$, we have, in $\tilde D_i\subset D_i$,
\begin{align*}
\partial_t \tilde w_i-\Delta \tilde w_i-K|\nabla\tilde w_i|
&\le e^{-\lambda t}\bigl[\partial_t w_i-\Delta w_i-K|\nabla w_i|-\lambda w_i\bigr] -\eps\bigl[\psi_t-\Delta\psi-K|\nabla\psi|\bigr]\\
&\le e^{-\lambda t}\bigl[-\lambda w_i+\sum_{j=1}^m c_{ij}(x,t)w_j\bigr] \\
&= \bigl[-\lambda \tilde w_i+\sum_{j=1}^m c_{ij}(x,t)\tilde w_j\bigr]+\eps\bigl[-\lambda +\sum_{j=1}^m c_{ij}(x,t)\bigr]\psi \\
&\le (c_{ii}(x,t)-\lambda) \tilde w_i+\sum_{{j=1 \atop j\ne i}}^m c_{ij}(x,t)\tilde w_j,
\end{align*}
by choosing $\lambda=\max_i\sum_{j=1}^m \|c_{ij}\|_\infty$. Noting that
$\tilde w_i<0$ for $|x|$ large, we may then apply the previous argument to get $\tilde W\le 0$ in $Q_T$.
The conclusion follows upon letting~$\eps\to 0$.
\smallskip

(ii) It suffices to apply assertion (i) to the functions $\hat w_i:=w_i-M$, noting
that, in view of assumption (\ref{condsum}), we have, for each $i=1,\dots,m$,
\begin{align*}\partial_t \hat w_i-\Delta \hat w_i-  K |\nabla \hat w_i|
&=\partial_t w_i-\Delta w_i-  K |\nabla w_i| \\
&\le \sum_{j=1}^m c_{ij}(x,t)(\hat w_j+M)\le \sum_{j=1}^m c_{ij}(x,t)\hat w_j
\end{align*}
in $D_i\supset\bigl\{(x,t)\in Q_T:\ \hat w_i(x,t)>0\bigr\}$.
\qed

\section*{Acknowledgments} The first author is supported  by Vietnam National Foundation for Science and Technology Development (NAFOSTED) under grant number 101.02-2014.06.
The second author is partially supported by the Labex MME-DII (ANR11-LBX-0023-01).


\begin{thebibliography}{10}

\bibitem{AA99}
Akhmediev, N. and Ankiewicz, A.,
\newblock Partially coherent solitons on a finite background.
\newblock {\em Phys. Rev. Lett.}, 82:2661--2664, 1999.

\bibitem{Ama85}
Amann, H.,
\newblock Global existence for semilinear parabolic systems.
\newblock {\em J. Reine Angew. Math.}, 360:47--83, 1985.

\bibitem{BDW10}
Bartsch, Th., Dancer, N. and Wang Z.-Q.,
\newblock A {L}iouville theorem, a-priori bounds, and bifurcating branches of
  positive solutions for a nonlinear elliptic system.
\newblock {\em Calc. Var. Partial Differential Equations}, 37(3-4):345--361, 2010.

\bibitem{Beb89}
Bebernes, J. and Eberly, D.,
\newblock {\em Mathematical problems from combustion theory}, volume~83 of {\em
  Applied Mathematical Sciences}.
\newblock Springer-Verlag, New York, 1989.

\bibitem{BV98}
Bidaut-V{\'e}ron, M.-F.,
\newblock Initial blow-up for the solutions of a semilinear parabolic equation
  with source term.
\newblock In {\em \'{E}quations aux d\'eriv\'ees partielles et applications},
  pages 189--198. Gauthier-Villars, \'Ed. Sci. M\'ed. Elsevier, Paris, 1998.

\bibitem{BVR96}
Bidaut-V{\'e}ron, M.-F. and Raoux, Th.,
\newblock Asymptotics of solutions of some nonlinear elliptic systems.
\newblock {\em Comm. Partial Differential Equations}, 21(7-8):1035--1086, 1996.

\bibitem{DWZ11}
Dancer, N., Wang, K. 
and Zhang Z.,
\newblock Uniform {H}\"older estimate for singularly perturbed parabolic
  systems of {B}ose-{E}instein condensates and competing species.
\newblock {\em J. Differential Equations}, 251(10):2737--2769, 2011.

\bibitem{Dan92}
Dancer, N.,
\newblock Some notes on the method of moving planes.
\newblock {\em Bull. Austral. Math. Soc.}, 46(3):425--434, 1992.

\bibitem{DWW10}
Dancer, N., Wei, J. and Weth, T.,
\newblock A priori bounds versus multiple existence of positive solutions for a
  nonlinear {S}chr\"odinger system.
\newblock {\em Ann. Inst. H. Poincar\'e Anal. Non Lin\'eaire}, 27(3):953--969, 2010.

\bibitem{DW12}
Dancer, N. and Weth, T.,
\newblock Liouville-type results for non-cooperative elliptic systems in a
  half-space.
\newblock {\em J. Lond. Math. Soc. (2)}, 86(1):111--128, 2012.

\bibitem{FP09}
F{\"o}ldes, J. and Pol{\'a}{\v{c}}ik, P.,
\newblock On cooperative parabolic systems: {H}arnack inequalities and
  asymptotic symmetry.
\newblock {\em Discrete Contin. Dyn. Syst.}, 25(1):133--157, 2009.

\bibitem{GS81a}
Gidas, B. and Spruck, J.,
\newblock Global and local behavior of positive solutions of nonlinear elliptic
  equations.
\newblock {\em Comm. Pure Appl. Math.}, 34(4):525--598, 1981.

\bibitem{GLW14}
Guo, Y.,  Li, B. and Wei, J.,
\newblock Entire nonradial solutions for non-cooperative coupled elliptic
  system with critical exponents in {$\mathbb{R}\sp 3$}.
\newblock {\em J. Differential Equations}, 256(10):3463--3495, 2014.

\bibitem{GL08}
Guo, Y. and Liu, J.,
\newblock Liouville type theorems for positive solutions of elliptic system in
  {$\mathbb R\sp N$}.
\newblock {\em Comm. Partial Differential Equations}, 33(1-3):263--284, 2008.

\bibitem{Hio99}
Hioe, F.~T.,
\newblock Solitary waves for $\mathit{N}$ coupled nonlinear schr\"odinger
  equations.
\newblock {\em Phys. Rev. Lett.}, 82:1152--1155, Feb 1999.

\bibitem{LLW15}
Liu, J., Liu, X. and Wang, Z.-Q.,
\newblock Multiple mixed states of nodal solutions for nonlinear
  {S}chr\"odinger systems.
\newblock {\em Calc. Var. Partial Differential Equations}, 52(3-4):565--586, 2015.

\bibitem{MZ08}
Ma, L. and Zhao, L.,
\newblock Uniqueness of ground states of some coupled nonlinear {S}chr\"odinger
  systems and their application.
\newblock {\em J. Differential Equations}, 245(9):2551--2565, 2008.

\bibitem{MZ00}
Merle F. and Zaag, H.,
\newblock A {L}iouville theorem for vector-valued nonlinear heat equations and
  applications.
\newblock {\em Math. Ann.}, 316(1):103--137, 2000.

\bibitem{MSS14}
Montaru, A., Sirakov, B. and Souplet, Ph.,
\newblock Proportionality of components, {L}iouville theorems and a priori
  estimates for noncooperative elliptic systems.
\newblock {\em Arch. Ration. Mech. Anal.}, 213(1):129--169, 2014.

\bibitem{Mor89}
Morgan, J.,
\newblock Global existence for semilinear parabolic systems.
\newblock {\em SIAM J. Math. Anal.}, 20(5):1128--1144, 1989.

\bibitem{Pha14}
Phan, Q.H.,
\newblock Optimal {L}iouville-type theorem for a parabolic system.
\newblock {\em Discrete Contin. Dyn. Syst.}, 35(1):399--409, 2015.

\bibitem{Pol09}
Pol{\'a}{\v{c}}ik, P.,
\newblock Symmetry properties of positive solutions of parabolic equations: a survey.
\newblock In {\em Recent progress on reaction-diffusion systems and viscosity
  solutions}, pages 170--208. World Sci. Publ., Hackensack, NJ, 2009.
  
\bibitem{PQ06}
Pol{\'a}{\v{c}}ik, P. and Quittner, P.,  
\newblock A Liouville-type theorem and the decay of radial solutions of a semilinear heat equation, 
\newblock {\em Nonlinear Anal.}, 64:1679--1689, 2006.

\bibitem{PQS07}
Pol{\'a}{\v{c}}ik, P., Quittner, P. and Souplet, Ph.,
\newblock Singularity and decay estimates in superlinear problems via
  {L}iouville-type theorems. {I}. {E}lliptic equations and systems.
\newblock {\em Duke Math. J.}, 139(3):555--579, 2007.

\bibitem{PQS07b}
Pol{\'a}{\v{c}}ik, P., Quittner, P. and Souplet, Ph.,
\newblock Singularity and decay estimates in superlinear problems via
  {L}iouville-type theorems. {II}. {P}arabolic equations.
\newblock {\em Indiana Univ. Math. J.}, 56(2):879--908, 2007.

\bibitem{PW67}
\newblock Protter, M. and Weinberger, H.,
\newblock {\em Maximum principles in differential equations}, 
Prentice Hall, Englewood Cliffs, N.J., 1967.

\bibitem{Qui14}
Quittner, P.,
\newblock Liouville theorems for scaling invariant superlinear parabolic
  problems with gradient structure.
\newblock {\em Math. Ann.}, DOI 10.1007/s00208-015-1219-7 (2015).

\bibitem{Qui15}
Quittner, P.,
\newblock Liouville theorems, universal estimates and periodic solutions for
  cooperative parabolic {L}otka-{V}olterra systems.
\newblock {\em Preprint}, 2015.

\bibitem{QS07}
Quittner, P. and Souplet, Ph.,
\newblock {\em Superlinear parabolic problems}.
\newblock Birkh\"auser Advanced Texts: Basler Lehrb\"ucher. [Birkh\"auser
  Advanced Texts: Basel Textbooks]. Birkh\"auser Verlag, Basel, 2007.
\newblock Blow-up, global existence and steady states.

\bibitem{QS11}
Quittner, P. and Souplet, Ph.,
\newblock Optimal {L}iouville-type theorems for noncooperative elliptic {S}chr\"odinger systems and applications.
\newblock {\em Comm. Math. Phys.}, 311(1):1--19, 2012.

\bibitem{QS12}
Quittner, P. and Souplet, Ph.,
\newblock Symmetry of components for semilinear elliptic systems.
\newblock {\em SIAM J. Math. Anal.}, 44(4):2545--2559, 2012.

\bibitem{RZ00}
Reichel, W. and Zou, H.,
\newblock Non-existence results for semilinear cooperative elliptic systems via
  moving spheres.
\newblock {\em J. Differential Equations}, 161(1):219--243, 2000.

\bibitem{TTVW}
Tavares, H., Terracini, S., Verzini, G. and Weth, T.,
\newblock Existence and nonexistence of entire solutions for non-cooperative
  cubic elliptic systems.
\newblock {\em Comm. Partial Differential Equations}, 36(11):1988--2010, 2011.

\bibitem{WW08}
Wei, J. and Weth, T.,
\newblock Radial solutions and phase separation in a system of two coupled
  {S}chr\"odinger equations.
\newblock {\em Arch. Ration. Mech. Anal.}, 190(1):83--106, 2008.

\end{thebibliography}
\end{document}